\documentclass[preprint,11pt]{elsarticle}

\usepackage{amsmath,amsthm,amssymb,tikz,patternmacros,subcaption}
\usepackage[hidelinks]{hyperref}
\usepackage[capitalise,noabbrev]{cleveref}
\usepackage{xcolor}
\usetikzlibrary{shapes,patterns}
\hypersetup{colorlinks}

\newtheorem{thm}{Theorem}\crefname{thm}{Theorem}{Theorems}
\newtheorem{lem}[thm]{Lemma}\crefname{lem}{Lemma}{Lemmas}
\newtheorem{ex}[thm]{Example}\crefname{ex}{Example}{Examples}
\newtheorem{cor}[thm]{Corollary}\crefname{cor}{Corollary}{Corollaries}
\newtheorem{conj}[thm]{Conjecture}\crefname{conj}{Conjecture}{Conjectures}
\newtheorem{prop}[thm]{Proposition}\crefname{prop}{Proposition}{Propositions}
\newtheorem{defn}[thm]{Definition}\crefname{defn}{Definition}{Definitions}
\newtheorem{rem}[thm]{Remark}\crefname{rem}{Remark}{Remarks}
\newtheorem{que}[thm]{Question}\crefname{que}{Question}{Questions}
\crefname{clm}{Claim}{Claims}
\crefname{figure}{Figure}{Figure}

\newcommand\cl[1]{#1_{cl}}
\newcommand\sh[1]{#1_{sh}}
\newcommand\occX[2]{\eta^{#2}_{#1}}
\newcommand\mX[2]{#1^{-}_{#2}}
\DeclareMathOperator{\muP}{\mu_\mathcal{P}}
\DeclareMathOperator{\muM}{\mu_\mathcal{M}}

\numberwithin{equation}{section}
\numberwithin{figure}{section}
\numberwithin{thm}{section}

\begin{document}
\begin{frontmatter}
\title{The Poset of Mesh Patterns}
\author{Jason P. Smith\fnref{1}}
\fntext[1]{This research was supported by the EPSRC Grant EP/M027147/1}
\author{Henning Ulfarsson\fnref{2}}
\fntext[2]{Research partially supported by grant 141761-051 from the Icelandic Research Fund}
\begin{abstract}
	We introduce the poset of mesh patterns, which generalises the permutation pattern poset.
	We fully classify the mesh patterns for which the interval~$[1^\emptyset,m]$ is non-pure,
	where $1^\emptyset$ is the unshaded singleton mesh
	pattern. We present some results on the M\"obius function of the poset, and show
	that~$\mu(1^\emptyset,m)$ is almost always zero. Finally, we introduce a class
	of disconnected and non-shellable intervals by generalising the direct product operation
	from permutations to mesh patterns.
	\end{abstract}
\end{frontmatter}

\section{Introduction}
Mesh patterns were first introduced by Br\"and\'en and Claesson in \cite{Bra11} as a generalisation
of permutation patterns, and have been studied extensively in recent years, see e.g.,~\cite{CTU15,JKR15}.
A mesh pattern consist of a pair $(\pi,P)$, where $\pi$ is a permutation and $P$ is a set of coordinates
in a square grid. For example, $(312,\{(0,0),(1,2)\})$ is a mesh
pattern, which we depict by
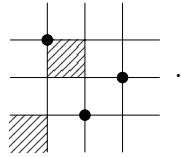
\begin{center}
\patt{0.5}{3}{3,1,2}[0/0,1/2][][][][][4].
\end{center}

A natural definition of when one mesh patterns occurs in another mesh patterns was given in~\cite{TU17},
which we present in \cref{sec:PosMP}.
This allows us to generalise the classical permutation poset to a poset of mesh patterns, where
$(\sigma,S)\le(\pi,P)$ if there is an occurrence of $(\sigma,S)$ in~$(\pi,P)$. The permutation poset
has received a lot of attention in recent years, but due to its
complicated structure a full understanding of it has proven elusive,
see~\cite{McSt13,Smith15}. The poset of mesh patterns, which we define
here, contains the poset of permutations as an induced subposet. Therefore, investigating
the poset of mesh patterns may lead to a better understanding of the poset of permutations. Moreover, 
studying this poset may help to answer some of the open questions on mesh patterns.

In \cref{sec:PosMP} we introduce the poset of mesh patterns and related definitions,
including a brief overview of poset topology. In \cref{sec:MF} we prove some results
on the M\"obius function of this poset. In \cref{sec:purity} we give
a characterisation of the non-pure (or non-ranked) intervals of the poset.
In \cref{sec:topology} we give some results on the topology of the poset.

\section{The Poset of Mesh Patterns}\label{sec:PosMP}
To define a mesh pattern we begin with a permutation
$\pi=\pi_1\pi_2\ldots\pi_n$. We can plot $\pi$ on an $n\times n$ grid,
where we place a dot at coordinates $(i,\pi_i)$, for all $1\le i\le n$.
A \emph{mesh pattern} is then obtained by shading some of the boxes of this grid, so a mesh
pattern takes the form $p=(\cl{p},\sh{p})$, where $\cl{p}$ is a permutation
and $\sh{p}$ is a set of coordinates recording the shaded boxes, which are indexed
by their south west corner. For ease of notation we sometimes denote the mesh
pattern $(\cl{p},\sh{p})$ as $\cl{p}^{\sh{p}}$. We let~$|\cl{p}|$ represent the length of $\cl{p}$
and $|\sh{p}|$ the size of $\sh{p}$, and define the \emph{length} of $p$ as $|\cl{p}|$, which we denote $|p|$.
For example, the mesh pattern $(132,\{(0,0),(0,1),(2,2)\})$,
or equivalently $132^{(0,0),(0,1),(2,2)}$, has the form:
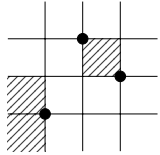
\begin{center}
\patt{0.5}{3}{1,3,2}[0/1,0/0,2/2][][][][][4]
\end{center}

To define when a mesh pattern occurs within another mesh pattern, we first need to
recall two other well-known definitions of occurrence. A permutation $\sigma$
\emph{occurs} in a permutation $\pi$ if there is a subsequence, $\eta$, of $\pi$ whose letters
appear in the same relative order of size as the letters of $\sigma$. The subsequence
$\eta$ is called an \emph{occurrence} of $\sigma$ in $\pi$. If no such occurrence exists
we say that $\pi$ \emph{avoids} $\sigma$.

Consider a mesh pattern $(\sigma,S)$ and an occurrence $\eta$ of $\sigma$ in $\pi$, in the
classical permutation pattern sense. Each box $(i,j)$ of $S$ corresponds to an
area $R_{\eta}(i,j)$ in the plot of $\pi$, which is the rectangle whose corners are the points in $\pi$
which in $\eta$ correspond to the letters $\sigma_i,\sigma_{i+1},j,j+1$ of $\sigma$, and the letters $\sigma_0,\sigma_{|\sigma|+1},0$ and $|\sigma|+1$ are to the south, north, east and west boundaries, respectively.
A point is contained in $R_\eta(i,j)$ if it is in the interior of $R_\eta(i,j)$, that is, not on the boundary.
For example, in \cref{fig:occEx} where $\eta$ is the occurence in red, the area
of $R_\eta(0,0)$ contains the boxes $\{(0,0),(1,0),(0,1),(1,1)\}$, and it contains exactly one point.
We say that $\eta$ is an occurrence of the mesh pattern $(\sigma,S)$ in the permutation~$\pi$
if there is no point in $R_{\eta}(i,j)$, for all shaded boxes~$(i,j)\in S$.

Using these definitions of occurrences we can recall a concept of mesh
pattern containment in another mesh pattern introduced in \cite{TU17}.
An example of which is given in \cref{fig:occEx}.

\begin{defn}[\cite{TU17}]\label{defn:meshOcc}
An occurrence of a mesh pattern $(\sigma,S)$ in another mesh
pattern $(\pi,P)$ is an occurrence~$\eta$ of~$(\sigma,S)$ in $\pi$, where for any $(i,j)\in S$
every box in $R_\eta(i,j)$ is shaded in $(\pi,P)$.
\end{defn}

\begin{figure}
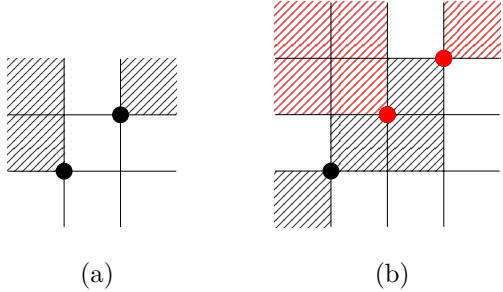
\centering
\begin{subfigure}[b]{0.3\textwidth}
\centering\patt{0.75}{2}{1,2}[0/1,0/2,2/2][][][][]
\caption{}\label{subfiga}\end{subfigure}
%\begin{subfigure}[b]{0.3\textwidth}
%\centering\patt{0.75}{3}{1,2,3}[0/0,0/2,0/3,1/3,1/1,1/2,2/1,2/2,3/3][][][][]
%\caption{}\label{subfigb}\end{subfigure}
\begin{subfigure}[b]{0.3\textwidth}\centering
\colpatt{0.75}{3}{1,2,3}[0/0,0/2,0/3,1/3,1/1,1/2,2/1,2/2,3/3][2/2,3/3][3/3,0/2,0/3,1/3,1/2]
\caption{}\label{subfigc}\end{subfigure}
\caption{A pair of mesh patterns, with an occurrence of (a) in~(b) depicted in red.
 %Also note (b) does not contain $12^{(0,0),(1,1),(2,2)}$.
 }\label{fig:occEx}
\end{figure}

The classical permutation poset $\mathcal{P}$ is defined as the poset of all permutations,
with $\sigma\le_\mathcal{P}\pi$ if and only if $\sigma$ occurs in $\pi$. Using \cref{defn:meshOcc}
we can similarly define the mesh pattern poset $\mathcal{M}$ as the poset of all mesh patterns,
with $m\le_\mathcal{M} p$ if $m$ occurs in $p$. We drop the subscripts from $\le$ when it is clear
which partial order is being considered. An \emph{interval} $[\alpha,\beta]$ of a poset is defined as the subposet induced by the set
$\{\kappa\,|\,\alpha\le\kappa\le\beta\}$. See \cref{fig:intEx} for an example of an interval of $\mathcal{M}$.

\begin{figure}\centering
\begin{tikzpicture}
\def\y{3}
\def\x{3}
\def\s{0.5}
\node (123-123) at (0*\x,5*\y){\patt{\s}{3}{1,2,3}[0/3,1/3,2/3][][][][][4]};
\node (123-12) at (-1*\x,4*\y){\patt{\s}{3}{1,2,3}[0/3,1/3][][][][][4]};
\node (123-13) at (0*\x,4*\y){\patt{\s}{3}{1,2,3}[0/3,2/3][][][][][4]};
\node (123-23) at (1*\x,4*\y){\patt{\s}{3}{1,2,3}[1/3,2/3][][][][][4]};
\node (123-1) at (-1.5*\x,3*\y){\patt{\s}{3}{1,2,3}[0/3][][][][][4]};
\node (12-12) at (-0.5*\x,3*\y){\patt{\s}{2}{1,2}[0/2,1/2][][][][][4]};
\node (123-2) at (0.5*\x,3*\y){\patt{\s}{3}{1,2,3}[1/3][][][][][4]};
\node (123-3) at (1.5*\x,3*\y){\patt{\s}{3}{1,2,3}[2/3][][][][][4]};
\node (12-1) at (-1*\x,2*\y){\patt{\s}{2}{1,2}[0/2][][][][][4]};
\node (123-0) at (0*\x,2*\y){\patt{\s}{3}{1,2,3}[][][][][][4]};
\node (12-2) at (1*\x,2*\y){\patt{\s}{2}{1,2}[1/2][][][][][4]};
\node (12-0) at (-1*\x,1*\y){\patt{\s}{2}{1,2}[][][][][][4]};
\node (1-1) at (1*\x,1*\y){\patt{\s}{1}{1}[0/1][][][][][4]};
\node (1-0) at (0*\x,0*\y){\patt{\s}{1}{1}[][][][][][4]};
\draw (123-123) -- (123-12);\draw (123-123) -- (123-13);\draw (123-123) -- (123-23);
\draw[-] (123-123) to [bend right=15] (12-12);\draw[-] (12-12) to [bend left=15] (1-1);
\draw (123-12) -- (123-1);\draw (123-12) -- (123-2);
\draw (123-13) -- (123-1);\draw (123-13) -- (123-3);
\draw (123-23) -- (123-2);\draw (123-23) -- (123-3);
\draw (123-1) -- (123-0);\draw (123-1) -- (12-1);
\draw (123-2) -- (123-0);
\draw (123-3) -- (123-0);\draw (123-3) -- (12-2);
\draw (12-12) -- (12-1);\draw (12-12) -- (12-2);
\draw (12-1) -- (12-0);\draw (12-2) -- (12-0);
\draw (123-0) -- (12-0);\draw (12-0) -- (1-0);
\draw (1-1) -- (1-0);
\end{tikzpicture}
\caption{The interval $[1^\emptyset,123^{(0,3),(1,3),(2,3)}]$ of $\mathcal{M}$.}\label{fig:intEx}
\end{figure}
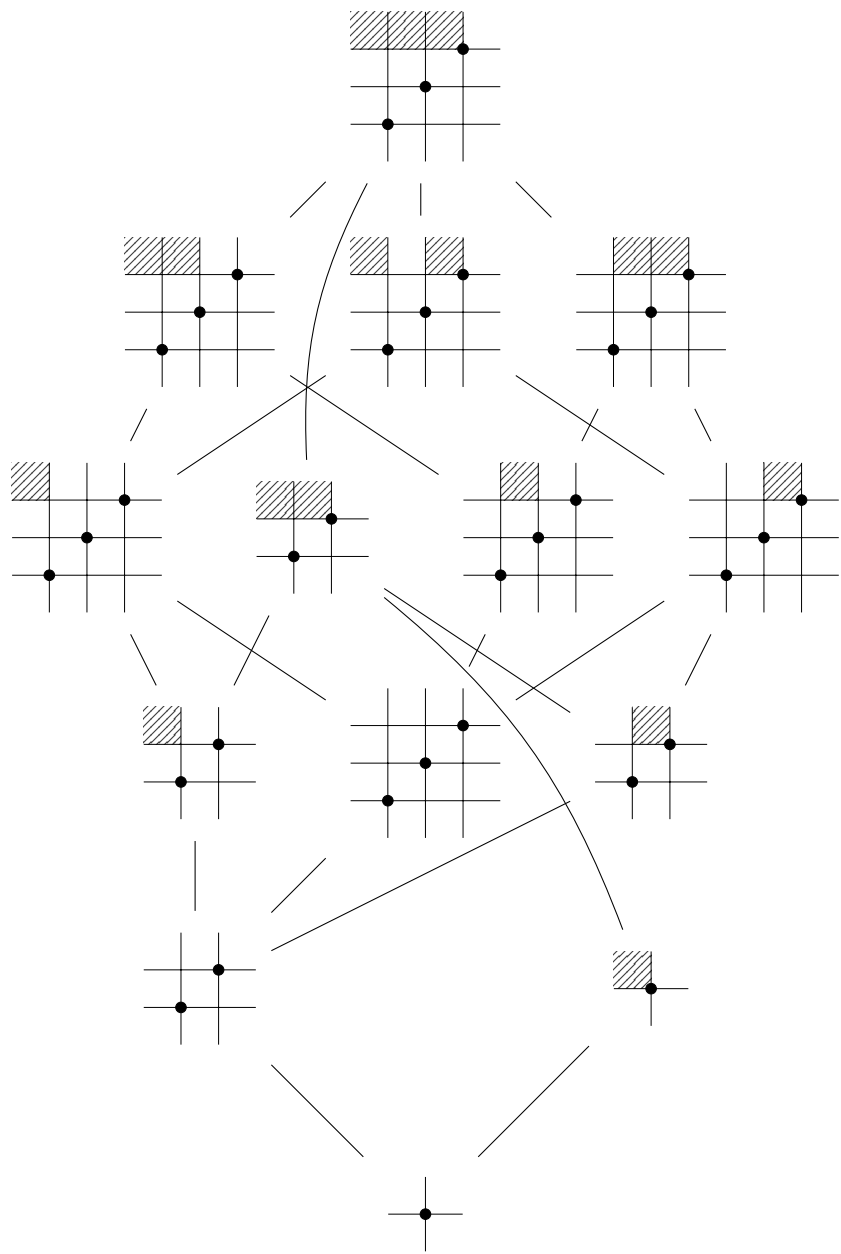

The first result on the mesh pattern poset is that there are infinitely many maximal
elements, which shows a significant difference to the permutation poset, where there are no maximal elements.

\begin{lem}\label{lem:max}
The poset of mesh pattern contains infinitely many maximal elements, which are the mesh
patterns in which all boxes are shaded.
\end{lem}

\begin{proof}
This follows from the easily proven fact that a fully shaded mesh pattern occurs only
in itself, and in no other mesh patterns.
\end{proof}

\subsection{Poset Topology}
In this subsection we briefly introduce some poset topology, and refer the reader to \cite{Wac07}
for a comprehensive overview of the topic, including any definitions we omit here.

The \emph{M\"obius function} of an interval $[\alpha,\beta]$ of a poset is defined by:\linebreak
${\mu(a,a)=1}$, for all $a$, $\mu(a,b)=0$ if $a\not\le b$, and $$\mu(a,b)=-\sum_{c\in[a,b)}\mu(a,c).$$
See \cref{fig:1-123} for an example. The M\"obius function of a poset $P$ is given by $\mu(P)=\mu(\hat{0},\hat{1})$, where $\hat{0}$ and $\hat{1}$ are unique minimal and maximal elements which we add to $P$. 

In a poset we say that $\alpha$ \emph{covers} $\beta$, denoted $\alpha\gtrdot\beta$, if $\alpha>\beta$
and there is no $\kappa$ such that $\alpha>\kappa>\beta$. A \emph{chain} of length $k$ in a poset is
a totally ordered subset $c_1<c_2<\cdots<c_{k+1}$, and the chain is \emph{maximal} if $c_i\lessdot c_{i+1}$,
for all $1\le i \le k$. A poset is \emph{pure} (also known as \emph{ranked}) if all
maximal chains have the same length. The \emph{dimension} of a poset $P$, denoted $\dim P$, is the
length of the longest maximal chain. For example, the interval in \cref{fig:intEx} is nonpure because
there is one maximal chain of length $3$ ($\pattin{1}{1}{}\lessdot\pattin{1}{1}{0/1}\lessdot
\pattin{2}{1,2}{0/2,1/2}\lessdot\pattin{3}{1,2,3}{0/3,1/3,2/3}$), two maximal chains of length $4$
and all other maximal chains have length $5$, so the interval has dimension $5$.

The \emph{interior} of an interval $[\alpha,\beta]$ is obtained by removing $\alpha$ and $\beta$,
and is denoted $(\alpha,\beta)$. The \emph{order complex} of an interval $[\alpha,\beta]$, denoted
$\Delta(\alpha,\beta)$, is the simplicial complex whose faces are the chains of $(\alpha,\beta)$.
When we refer to the \emph{topology} of an interval we mean the topology of the order complex of the interval.

A simplicial complex is \emph{shellable} if we can order the maximal faces $F_1,\ldots,F_t$
such that the subcomplex $\left(\cup_{i=1}^{k-1}F_i\right)\cap F_k$ is pure and
$(\dim F_k)$-dimensional, for all $k=2,\ldots,t$. Being shellable implies other properties
on the topology, such as having the homotopy type of a wedge of spheres.

An interval $I$ is \emph{disconnected} if the interior can be split into two disjoint pairwise incomparable sets,
that is, $I=A\cup B$ with $A\cap B=\emptyset$ and for every~$a\in A$
and $b\in B$ we have $a\not\le b$ and $b\not\le a$.
Each interval $I$ can be decomposed into its smallest connected parts, which we call the \emph{components} of $I$.
A component is \emph{nontrivial} if it contains more than one element
and we say an interval is \emph{strongly disconnected} if it has at least two nontrivial components.
For example, the interval $[1^\emptyset,12^{(0,2),(1,2)}]$ in \cref{fig:intEx} is disconnected
but not strongly disconnected. Note that if an interval has dimension less than $3$ it can
never be strongly disconnected.

We can use disconnectivity as a test for shellability using the following results.

\begin{lem}\label{lem:strongdis}
If an interval is strongly disconnected, then it is not shellable.
\begin{proof}
Consider any ordering of the maximal chains and let $F_k$, with ${k>1}$, be the first chain where
every preceding chain belongs to a different component and $F_k$ belongs to a
nontrivial component. Note that such an $F_k$ exists in every ordering because the interval
is strongly disconnected, and because $F_k$ belongs to a nontrivial component it must have dimension
of at least $1$. So $\left(\cup_{i=1}^{k-1}F_i\right)\cap F_k=\emptyset$,
which has dimension $-1$, so it is not $\dim(F_k-1)$-dimensional. Therefore, the ordering is not a
shelling.
\end{proof}
\end{lem}

Since every subinterval of a shellable interval is shellable, \cite[Corollary 3.1.9]{Wac07},
we obtain the following:

\begin{cor}
An interval which contains a strongly disconnected subinterval is not shellable.
\end{cor}

Finally, we present a useful result known as the Quillen Fiber Lemma \cite{Quillen78}. Two simplicial
complexes are homotopy equivalent  if one can be obtained by deforming the other but
not breaking or creating any new ``holes", for a formal definition see \cite{Hat02}.
A simplicial complex is \emph{contractable} if it is
homotopy equivalent to a point and if two posets are homotopy equivalent their
M\"obius functions are equal.  Given a poset $P$, with $p \in P$ define the upper
ideal $P_{\ge p}=\{q\in P\,|\,q\ge p\}$.

\begin{prop}\label{thm:Quil}(Quillen Fiber Lemma)
Let $\phi:P\rightarrow Q$ be an order-preserving map between posets such that for any
$x\in Q$ the complex\linebreak $\Delta(\phi^{-1}(Q_{\ge x}))$ is contractible.
Then $P$ and $Q$ are homotopy equivalent.
\end{prop}

\section{M\"obius Function}\label{sec:MF}
In this section we present some results on the M\"obius function of the mesh pattern poset.
We begin with some simple results on: mesh patterns with the same underlying permutations;
the mesh patterns with no points~$\epsilon^\emptyset$ and $\epsilon^{(0,0)}$;
and mesh patterns with no shaded boxes.
Throughout the remainder of the paper we assume that $m$ and $p$ are
mesh patterns.

\begin{lem}
Let $\pi$ be a permutation. For any sets $A\subseteq B$ the interval~$[\pi^A,\pi^B]$ is isomorphic to the boolean
 lattice~$B_{|B|-|A|}$. Therefore,\linebreak ${\mu(\pi^A,\pi^B)=(-1)^{|B|-|A|}}$ and
$[\pi^A,\pi^B]$ is shellable.
\begin{proof}
The elements of $[\pi^A,\pi^B]$ are exactly the mesh patterns $\pi^C$
where $C\subseteq B\setminus A$, which implies the result.
\end{proof}
\end{lem}

\begin{lem}
Consider $A\in\{\emptyset,(0,0)\}$, then:
$$\mu(\epsilon^{A},p)=\begin{cases}
1,&\mbox{ if }p=\epsilon^A \\
-1,&\mbox{ if }A=\emptyset\,\,\&\,\,|\cl{p}|+|\sh{p}|=1\\
0,&\mbox{ otherwise}
\end{cases}.$$
\begin{proof}
The first two cases are trivial. By the proof of \cref{lem:max} we know that
$\epsilon^{(0,0)}$ is not contained in any larger mesh patterns, which implies
$\mu(\epsilon^{(0,0)},p)=0$, for all $p\not=\epsilon^{(0,0)}$. If $|\cl{p}|+|\sh{p}|>1$, then
$(\epsilon^\emptyset,p)$ contains a unique minimal element $1^\emptyset$, so
$\mu(\epsilon^\emptyset,p)=0$.
\end{proof}
\end{lem}

\begin{lem}
The interval $[\sigma^\emptyset,\pi^\emptyset]$ is isomorphic to
$[\sigma,\pi]$ in $\mathcal{P}$, so $$\muM(\sigma^\emptyset,\pi^\emptyset)=\muP(\sigma,\pi).$$
\end{lem}

The M\"obius function of the classical permutation poset is known to be
unbounded \cite{Smith13}. So we get the following corollary:

\begin{cor}
The M\"obius function is unbounded on $\mathcal{M}$.
\end{cor}

We can also show that the M\"obius function is unbounded if we include shaded boxes.
We do this by mapping to the poset $\mathcal{W}$ of words with subword order,  that is, 
the poset made up of all words and $u\le w$ if there is a subword of $w$
that equals $u$. The map we introduce is analogous to the map in \cite[Section 2]{Smith14},
which maps certain intervals of the permutation poset to intervals of $\mathcal{W}$.
A \emph{descent} in a permutation $\pi=\pi_1\pi_2\ldots\pi_n$ is a pair of
letters $\pi_i,\pi_{i+1}$ with $\pi_{i}>\pi_{i+1}$. We call $\pi_{i+1}$ the \emph{descent bottom}.
An \emph{adjacency tail} is a letter $\pi_i$ with $\pi_i=\pi_{i-1}\pm 1$.
Let $adj(\pi)$ be the number of adjacency tails in~$\pi$.
Consider the set $\Gamma$ of mesh patterns where the permutation has exactly one descent,
the descent bottom is $1$ and we shade everything south west of~$1$. For example, the mesh pattern $2314^{(0,0),(1,0),(2,0)}$:
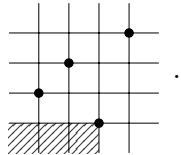
\begin{center}
\patt{0.4}{4}{2,3,1,4}[0/0,1/0,2/0][][][][][4].
\end{center}

\begin{lem}\label{lem:mobUn}
Consider a mesh pattern $m\in \Gamma$, then $[21^{(0,0),(1,0)},m]$ is shellable and
\[
\mu(21^{(0,0),(1,0)},m)=\begin{cases}
(-1)^{|m|}\lfloor\frac{|m|}{2}\rfloor,&\text{ if } adj(\cl{m})=0\\
(-1)^{|m|},& \text{ if } adj(\cl{m})=1 \text{ \& tail before descent}\\
0, &\text{ otherwise}
\end{cases}.
\]
\begin{proof}
First note that every mesh pattern in $[21^{(0,0),(1,0)},m]$ is also in $\Gamma$.
We define a map $f$ from $\Gamma$ to binary words in the following way. Let $b(x)$ be the
set of letters that appear before $1$ in $x\in \Gamma$. Set $\hat{f}(x)$ as the word where
the $i$th letter is $0$ if it is in $b(x)$ and $1$ otherwise, and let $f(x)$
equal $\hat{f}(x)$ with the first letter removed.
So $f(\Gamma)$ is the set of binary words with at least one~$0$. The inverse of this map
is obtained by the following procedure: 1) take a binary word $w\in f(\Gamma)$ and prepend a $1$;
2) put the positions that are $0$'s in increasing order followed by 
the positions that are $1$ in increasing order; and 4) shade everything southwest of $1$. 
So $f$ is a bijection.

It is straightforward
to check that $f$ is order preserving. So
the interval~$[21^{(0,0),(1,0)},m]$ is isomorphic to $[0,f(m)]$ in $\mathcal{W}$.
It was shown in
\cite{Bjo90} that intervals of $\mathcal{W}$
are shellable, which proves the shellability part. It was also shown that the M\"obius
function equals the number of normal
occurrences with the sign given by the dimension,
where an occurrence is \emph{normal} if in any consecutive sequence
of equal elements every
non-initial letter is part of the occurrence.
So for an occurrence of $0$ in $f(m)$ to be normal there can be no $1$ directly preceded by a $1$ and at most one
$0$ directly preceded by a $0$. If such a $0$ exists it must be the occurrence,
otherwise any $0$ can be the occurrence. In our bijection a non-initial letter 
of such a sequence maps to an adjacency tail.  Combining this with the fact that if there are no adjacency tails, then the letters before the descent must be all the even letters of which there are $\lfloor\frac{|m|}{2}\rfloor$, completes the proof.
\end{proof}
\end{lem}

The M\"obius function on $\mathcal{P}$ often takes larger values than on $\mathcal{M}$,
but it is not always true that $\muM(m,p)\le
\muP(\cl{m},\cl{p})$. A simple counterexample is the interval
$$[1^{(0,1)},123^{(0,2),(0,3),(1,2),(1,3)}],$$ which has M\"obius function $1$,
however $\muP(1,123)=0$, see \cref{fig:1-123}.

\begin{figure}\centering
\begin{tikzpicture}
\def\s{0.35}
\node (123) at (0,3){\footnotesize\textcolor{white}{1}\patt{\s}{3}{1,2,3}[0/2,0/3,1/2,1/3][][][][][4]\textcolor{red}{1}};
\node (12a) at (-1,1.5){\footnotesize\textcolor{red}{-1}\patt{\s}{2}{1,2}[0/2,1/2][][][][][4]\textcolor{white}{-1}};
\node (12b) at (1,1.5){\footnotesize\textcolor{white}{-1}\patt{\s}{2}{1,2}[0/1,0/2][][][][][4]\textcolor{red}{-1}};
\node (1) at (0,0){\footnotesize\textcolor{white}{-1}\patt{\s}{1}{1}[0/1][][][][][4]\textcolor{red}{1}};
\draw (1) -- (12a) -- (123) -- (12b) -- (1);
\node (123) at (5,3){{\footnotesize\textcolor{red}{0}} $123$ \textcolor{white}{0}};
\node (12) at (5,1.5){{\footnotesize\textcolor{red}{-1}} $12$ \textcolor{white}{-1}};
\node (1) at (5,0){{\footnotesize\textcolor{red}{1}} $1$ \textcolor{white}{1}};
\draw (1) -- (12) -- (123);
\end{tikzpicture}
\caption{The interval $[1^{(0,1)},123^{(0,2),(0,3),(1,2),(1,3)}]$
(left) in $\mathcal{M}$ and $[1,123]$ (right) in $\mathcal{P}$,
with the M\"obius function in red.}\label{fig:1-123}
\end{figure}
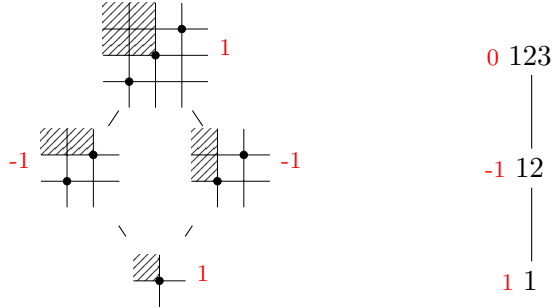

If we consider intervals where the bottom mesh pattern has no shadings, then we
get the following result:
\begin{lem}\label{lem:mu0}
Consider an interval $[s^\emptyset,p]$ in $\mathcal{M}$ with $\sh{p}\not=\emptyset$.
If $s^B\not\in(s^\emptyset,p)$ for any set $B$, then $\mu(s^\emptyset,p)=0$.
\begin{proof}
Consider the map $f:(s^\emptyset,p)\rightarrow A:x\mapsto\cl{x}^\emptyset,$ that is, $f$ removes all shadings from~$x$.
We can see that ${A=(s^\emptyset,\cl{p}^{\emptyset}]}$, so  $A$ is contractible, because it has the
unique maximal element~$\cl{p}^\emptyset$, hence $\mu(A)=0$. Moreover,~$f^{-1}(A_{\ge y})=[y,p)$,
for all $y\in A$, which is contractible. Therefore,~$(s^\emptyset,p)$ is homotopy equivalent
to $A$ by the Quillen Fiber Lemma (\cref{thm:Quil}), which implies~$\mu(s^\emptyset,p)=0$.
\end{proof}
\end{lem}
\begin{ex}
Consider the subinterval $[1^\emptyset,12^{(0,2)}]$ in \cref{fig:intEx},
applying \cref{lem:mu0} implies $\mu(1^\emptyset,12^{(0,2)})=0$. However, we cannot
apply \cref{lem:mu0} to $[1^\emptyset,12^{(0,2),(1,2)}]$ because it contains the
element $1^{(0,1)}$.
\end{ex}

We can combine Lemma~\ref{lem:mu0} with the following result to see that the
M\"obius function is almost always zero on the interval $[1^\emptyset,p]$.

\begin{lem}
As $n$ tends to infinity the proportion of mesh patterns of length~$n$ that contain any of
$\{1^{(0,0)},1^{(1,0)},1^{(0,1)},1^{(1,1)}\}$ approaches $0$.
\begin{proof}
Let $P(n,i)$ be the probability that the letter $i$ is an occurrence of~$1^{(0,0)}$
in a length $n$ mesh pattern, and let $P(n)$ be the probability
that a length $n$ mesh pattern contains $1^{(0,0)}$.

The probability $P(n,i)$ can be bounded
above by first considering the index $k$ of $i$, each having probability
$\frac{1}{n}$, and then requiring
that all boxes south west of $i$ are filled, of which there are $ik$. This
provides an upper bound, because it is possible that there is a
point south west of $i$, which would imply $i$ is not an occurrence of
$1^{(0,0)}$. We can formulate this as:
\begin{align*}
P(n,i)&\le\sum_{k=1}^{n}\frac{1}{n}\left(\frac{1}{2^i}\right)^k
=\frac{1}{n}\left(\frac{1-2^{-i(n+1)}}{1-2^{-i}}-1\right)\\
&=\frac{1}{n}\left(\frac{2^{-i}-2^{-i(n+1)}}{1-2^{-i}}\right)
=\frac{1}{n2^i}\left(\frac{1-2^{-in}}{1-2^{-i}}\right)
\le\frac{2}{n2^i}
\end{align*}

To compute the probability $P(n)$ we
can sum over all the $P(n,i)$.
Note again this is an over estimate because if a mesh pattern contains multiple
occurrences of $1^{(0,0)}$ it counts that mesh pattern more than once.

$$
P(n)\le\sum_{i=1}^{n}P(n,i)\le\sum_{i=1}^{n}\frac{2}{n2^i}
=\frac{2}{n}\left(\frac{1-\left(\frac{1}{2}\right)^{n+1}}{1-\frac{1}{2}}-1\right)
\le\frac{2}{n}
$$

Repeating this calculation for the other three shadings of $1$ implies that
the probability of containing any of the forbidden mesh patterns is bounded
by $\frac{8}{n}$ which tends to zero as $n$ tends to infinity.
\end{proof}
\end{lem}

Because of the previous lemma we obtain:

\begin{cor}
As $n$ tends to infinity the proportion of mesh patterns $p$ of length n such
that $\mu(1^\emptyset,p)=0$ approaches $1$.
\end{cor}

In the classical case it is true that given a permutation $\sigma$ the
probability a permutation of length $n$ contains $\sigma$ tends to $1$ as $n$
tends to infinity, this follows from the Marcus-Tardos Theorem \cite{MT04}. By
the above result we can see the same is not true in the mesh pattern case. In
fact we conjecture the opposite is true:

\begin{conj}
Given a mesh pattern $m$, with at least one shaded box, the probability that a random mesh pattern of length
$n$ contains $m$ tends to~$0$ as $n$ tends to infinity.
\end{conj}

\section{Purity}\label{sec:purity}
Recall that a poset is pure (also known as ranked) if all the maximal chains have the same length, and as we
can see from \cref{fig:intEx}, intervals of the mesh pattern poset can be non-pure. In this section we classify
which intervals~$[1^\emptyset,m]$ are non-pure. First we consider the length of the longest maximal chain in
any interval~$[1^\emptyset,m]$, that is, the dimension of $[1^\emptyset,m]$.

\begin{lem}
For any mesh pattern $m$, we have $\dim(1^\emptyset,m)=|\cl{m}|+|\sh{m}|$.
\begin{proof}
We can create a chain from $m$ to $1^\emptyset$ by deshading all boxes, in any order,
and then deleting all but one point, in any order. The length of this chain is $|\cl{m}|+|\sh{m}|$.
Moreover, we cannot create a longer chain because at every
step of a chain we must deshade a box or delete a point.
\end{proof}
\end{lem}

Therefore, we define the \emph{dimension} of a mesh pattern as $\dim(m)=|\cl{m}|+|\sh{m}|$ and we say an
edge $m\lessdot p$ is \emph{impure} if $\dim(p)-\dim(m)>1$.
Next we give a classification of impure edges.

 Let $\mX{m}{x}$ be the mesh pattern obtained by deleting the
point $x$ in $m$ and let $\occX{m}{x}$ be the occurrence of $\mX{m}{x}$ in $m$ that does not
use the point $x$. An occurrence $\eta$ of $m$ in $p$ \emph{uses the shaded box $(a,b)\in\sh{p}$}
if $(a,b)\in R_\eta(i,j)$ for some shaded box $(i,j)\in\sh{m}$. We say that
deleting a point $x$ \emph{merges shadings} if
there is a shaded box in $\mX{m}{x}$ that corresponds to more than one shaded box in
$\occX{m}{x}$, see \cref{fig:impEx}.

\begin{lem}\label{lem:impureEdge}
Two mesh patterns $m<p$ form an impure edge if and only if all occurrences of $m$ in
$p$ use all the shaded boxes of $p$ and are obtained by deleting a point that merges
shadings.
\end{lem}

\begin{proof}
First we show the backwards direction. Because $m$ is obtained by deleting a
point that merges shadings, $m$ must have one less point and at least one less
shaded box so $\dim(p)-\dim(m)\ge2$. So it suffices to show that there is no $z$ such
that $m<z<p$. Suppose such a $z$ exists, then if~$z$ is obtained by deshading a
box in $p$ it can no longer contain $m$ because all occurrences of $m$ in $p$
use all the shaded boxes of $p$. If $z$ is obtained by deleting a point and $m<z$,
then $\cl{m}=\cl{z}$. Therefore, we can deshade some boxes
of $z$ to get $m$, which implies there is an occurrence of $m$ in $p$ that
does not use all the shaded boxes of $p$.

Now consider the forward direction. Suppose $m\lessdot p$ is impure, so
$\dim(p)-\dim(m)\ge2$. Therefore, $m$ is obtained by deleting a single point
which merges shadings, but does not delete shadings because any other
combination of deleting points and deshading can be done in successive steps.
Furthermore, this must be true for any point that can be deleted to get $m$,
that is, for all occurrences of $m$ in $p$. Moreover, if there is an occurrence
that does not use all the shaded boxes of $p$, we can deshade the box it doesn't
use and get an element that lies between $m$ and $p$.
\end{proof}

\begin{figure}
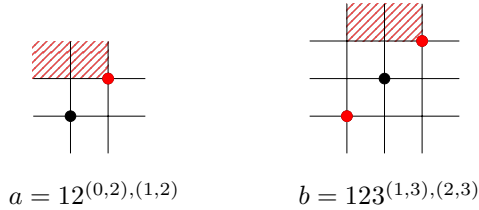
\centering
\begin{subfigure}[b]{0.3\textwidth}
\centering\colpatt{0.5}{2}{1,2}[0/2,1/2][2/2][0/2,1/2]
\caption*{$a=12^{(0,2),(1,2)}$}\label{subfig:a}\end{subfigure}
\begin{subfigure}[b]{0.3\textwidth}
\centering\colpatt{0.5}{3}{1,2,3}[1/3,2/3][1/1,3/3][1/3,2/3]
\caption*{$b=123^{(1,3),(2,3)}$}\label{subfig:b}\end{subfigure}
\caption{Two mesh patterns with a point $x$ in black whose deletion merges shadings and the occurrences
$\occX{a}{x}$ and $\occX{b}{x}$ in red. By \cref{lem:impureEdge} $\mX{a}{x}\lessdot a$ is impure,
but $\mX{b}{x}<b$ is not an impure edge because there is a second occurrence of $\mX{b}{x}$ in $b$, using points $23$,
that does not use all the shaded boxes in $b$.}\label{fig:impEx}
\end{figure}

\begin{lem}\label{lem:topImpure}
If $[m,p]$ contains an impure edge, then it contains an impure edge
$a\lessdot b$ where $\cl{p}=\cl{b}$.
\begin{proof}
Let $x\lessdot y$ be an impure edge in $[m,p]$. So $x$ is obtained from $y$
by deleting a point $i$. Consider an occurrence
$\eta$ of $y$ in $p$ and let $b$ be the mesh pattern where $\cl{b}=\cl{p}$
and $\sh{b}$ are the shaded boxes used by $\eta$. Let $a$ be the mesh pattern obtained from $b$
by deleting the point which corresponds to $i$ in $\eta$.

The mesh pattern $b$ is constructed from $y$ by adding a collection of points.
None of these added points can be touching a shaded box in $b$, as they must be added to empty
boxes of~$y$. Moreover, the set of occurrences of~$a$ in $b$ correspond to the set
of occurrences of $x$ in $y$, after adding the new points. This implies that 
the occurrences of $x$ in $y$ satisfy the conditions of \cref{lem:impureEdge} if and only if the
occurrences of~$a$ in $b$ satisfy the same conditions. So \cref{lem:impureEdge} implies $a\lessdot b$ is an
impure edge.
\end{proof}
\end{lem}

\begin{prop}
The interval $[1^\emptyset,m]$ is non-pure if and
only if there exists a point $x$ in $m$ whose deletion merges shadings and
there is no other occurrence of $\mX{m}{x}$ in $m$ which uses a subset of the shadings
used by $\occX{m}{x}$.
\begin{proof}
First we show the backwards direction. Let $t$  be the mesh pattern obtained by
inserting~$x$ back into $\mX{m}{x}$, and $\phi$ the corresponding occurrence of $\mX{m}{x}$
in $t$. Note that there are no other occurrences of $\mX{m}{x}$ in $t$ because there is
no occurrence of $\mX{m}{x}$ in $m$ which uses a subset of the shadings
used by $\occX{m}{x}$. Therefore, by Lemma~\ref{lem:impureEdge} we get that
$\mX{m}{x}\lessdot t$ is an impure edge.

To see the other direction suppose there is an impure edge in $[1^{\emptyset},m]$. By
Lemma~\ref{lem:topImpure} there is an impure edge $a\lessdot b$ where
$\cl{b}=\cl{m}$. By \cref{lem:impureEdge} all occurrences of $a$ in
$b$ use all shaded boxes of $b$ and are obtained by deleting a point that merges
shadings. Moreover, if deleting a point merges shadings in $b$, then its deletion merges shadings in $m$,
which implies the result.
\end{proof}
\end{prop}

\begin{cor}
There is an impure edge in the interval $[m,p]$ if and only if there exists a
point $x$ in $p$ whose deletion merges shadings and there is no other
occurrence of $\mX{p}{x}$ in $p$ with a subset of shadings of $\occX{p}{x}$, and
$\mX{p}{x}\ge m$.
\end{cor}

Note that containing an impure edge in $[m,p]$ does not necessarily imply
that $[m,p]$ is non-pure. For example, if $[m,p]$ contains only one edge and
that edge is impure, then $[m,p]$ is still pure. Although it is also possible
to have a pure poset that contains impure and pure edges, see \cref{fig:pureIm}.

\begin{figure}\centering
\begin{tikzpicture}
\def\y{1.35}
\def\x{2}
\def\s{0.25}
\node (2413) at (0*\x,3*\y){\patt{\s}{4}{2,4,1,3}[0/0,1/0,2/0][][][][][4]};
\node (213) at (-1*\x,1.15*\y){\patt{\s}{3}{2,1,3}[0/0,1/0][][][][][4]};
\node (312) at (0*\x,1.15*\y){\patt{\s}{3}{3,1,2}[0/0,1/0][][][][][4]};
\node (231) at (1*\x,1.85*\y){\patt{\s}{3}{2,3,1}[0/0,1/0,2/0][][][][][4]};
\node (21) at (0*\x,0*\y){\patt{\s}{2}{2,1}[0/0,1/0][][][][][4]};
\draw[-] (21) to [bend left=15] (213);
\draw[-] (21) to (312);
\draw[-] (21) to [bend right=15] (231);
\draw[-] (213) to [bend left=15] (2413);
\draw[-] (312) to (2413);
\draw[-] (231) to [bend right=15] (2413);
\end{tikzpicture}
\caption{The interval $[21^{(0,0),(1,0)},2413^{(0,0),(1,0),(2,0)}]$,
 which is pure but contains both pure and impure edges.}
 \label{fig:pureIm}
\end{figure}
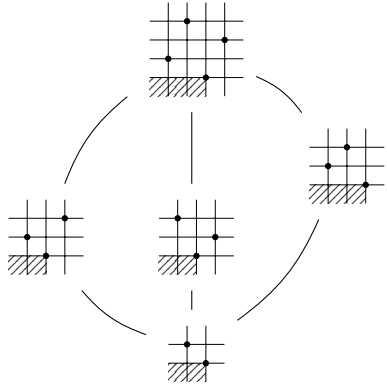

\section{Topology}\label{sec:topology}
A full classification of shellable intervals has not been obtained for the classical permutation
poset, so finding such a classification for the mesh pattern poset would be equally difficult, if not more so.
However, in \cite{McSt13} all disconnected intervals of the permutation poset are described, and containing a disconnected subinterval
implies a pure interval is not shellable. So this gives a large class of non-shellable intervals, in fact
it is shown that almost all intervals are not shellable. We showed in \cref{lem:strongdis} that containing a strongly
disconnected interval implies an interval is not shellable. So in this section we consider
when an interval is strongly disconnected. Firstly we look at the relationship between connectivity
in $\mathcal{P}$ and $\mathcal{M}$.

The connectivity of the interval $[\cl{m},\cl{p}]$ in $\mathcal{P}$ does not
necessarily imply the same property for $[m,p]$ in $\mathcal{M}$. For example,
the interval $[123,456123]$ is disconnected in $\mathcal{P}$ but the interval
\begin{equation}\label{eq:a}\left[\patt{.25}{3}{1,2,3}[3/0,3/1,3/2][][][][][4]
,\patt{.25}{6}{4,5,6,1,2,3}[6/0,6/1,6/2][][][][][4]\right]\end{equation} 
is a chain in $\mathcal{M}$, so is connected. Furthermore, the interval
$[321,521643]$ is connected in $\mathcal{P}$ but the interval
\begin{equation}\label{eq:b}\left[\patt{.25}{3}{3,2,1}[1/3][][][][][4],
\patt{.25}{6}{5,2,1,6,4,3}[1/5,1/6,4/6][][][][][4]\right]\end{equation}
is strongly disconnected in~$\mathcal{M}$. Therefore, if $[\cl{m},\cl{p}]$
is (non-)shellable in $\mathcal{P}$, then it is not true that $[m,p]$ has
the same property in $\mathcal{M}$. For example, $[123,456123]$ is not
shellable but~\eqref{eq:a} is shellable, and $[321,521643]$ is shellable
but \eqref{eq:b} is not shellable.

In \cite{McSt13} the direct sum operation is used to show that almost all
intervals of the permutation poset are not shellable in $\mathcal{P}$. We generalise the direct sum operation to
mesh patterns. Given two permutations $\alpha=\alpha_1\ldots\alpha_a$
and~$\beta=\beta_1\ldots\beta_b$ the direct sum of the two is defined as
$\alpha\oplus\beta=\alpha_1\ldots\alpha_a(\beta_1+a)(\beta_2+a)\ldots(\beta_b+a)$,
that is, we increase the value of each letter of $\beta$ by the
length of $\alpha$ and append it to $\alpha$. This can also be thought of in terms of the plots of $\alpha$
and $\beta$ by placing a copy of $\beta$ to the north east of $\alpha$.
Similarly we can define the skew-sum $\alpha\ominus\beta$ by
prepending $\alpha$ to $\beta$ and increasing the value of each letter of $\alpha$
by the length of $\beta$. We extend these definitions to mesh patterns in the following way:

\begin{defn}\label{defn:directsum}
Consider two mesh patterns $s$ and $t$, where the top right corner of $s$
and bottom left corner of $t$ are not shaded. The direct sum~$s\oplus t$ has the classical pattern
$\cl{s}\oplus\cl{t}$ and shaded boxes $\sh{s}\cup\{(i+|\cl{s}|,j+|\cl{s}|)\,|\,(i,j)\in\sh{t}\}$,
and also for any shaded boxes $(i,|\cl{s}|)$, $(|\cl{s}|,i)$, $(j,|\cl{s}|)$ or $(|\cl{s}|,j)$,
shaded all the boxes north, east, south or west of the box, respectively,
for all $0\le i< |\cl{s}|$ and $|\cl{s}|< j\le |\cl{s}|+|\cl{t}|$.  We similarly
define the skew-sum for when the bottom right corner of $s$ and top left corner of $t$ are not shaded.
\end{defn}

The direct product $s\oplus t$ can be consider as placing a copy of $t$
north east of $s$ and any shaded box that was on a boundary we extend to the new boundary,
see \cref{fig:directsum}. We define the direct sum in this way because it maintains one of the most
important properties in the permutation sense, that the first $|\cl{s}|$ letters are
an occurrence of $s$ and the final $|\cl{t}|$ letters are an occurrence of $t$.

A permutation is said to be indecomposable if it cannot be written as the direct
sum of smaller permutations. We generalise this to mesh patterns.
\begin{defn}
A mesh pattern $m$ is \emph{indecomposable} (resp. \emph{skew-\linebreak indecomposable}) if it
cannot be written $m=a\oplus b$ (resp. $m=a\ominus b$), where neither $a$ nor $b$ is $m$.
\end{defn}
\begin{rem}
It is well known that a permutation has a unique decomposition into indecomposable permutations.
This implies that a mesh pattern also has a unique decomposition.
\end{rem}

Using these definitions we can give a large class of strongly disconnected intervals,
which is a mesh pattern generalisation of Lemma 4.2 in \cite{McSt13}.

\begin{figure}
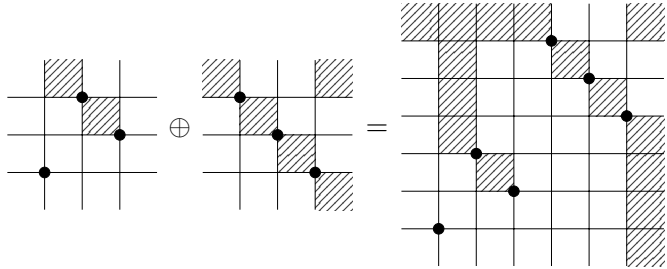
\centering
$\patt{.5}{3}{1,3,2}[1/3,2/2]\oplus\patt{.5}{3}{3,2,1}[0/3,1/2,2/1,3/3,3/0][][][][][4]
=\patt{.5}{6}{1,3,2,6,5,4}[1/3,2/2,3/6,4/5,5/4,6/6,6/3,1/4,1/5,1/6,0/6,2/6,6/0,6/1,6/2][][][][][4]$
\caption{The direct sum of two mesh patterns.}\label{fig:directsum}
\end{figure}

\begin{lem}
If $m$ is indecomposable, $\dim m > 1$ and $(0,0),(|m|,|m|)\not\in\sh{m}$, then
$[m,m\oplus m]$ is strongly disconnected.
\begin{proof}
By Lemma 4.2 in \cite{McSt13} the interval $[\cl{m},\cl{m}\oplus \cl{m}]$ is strongly disconnected,
with components $P_1=\{\cl{m}\oplus x\,|\,x\in [1,\cl{m})\}$ and
$P_2=\{x\oplus \cl{m}\,|\,x\in [1,\cl{m})\}$. Consider any pair $\alpha,\beta\in[m,m\oplus m]$,
if $\cl{\alpha}$ and~$\cl{\beta}$ are not in the same component of  $[\cl{m},\cl{m}\oplus \cl{m}]$,
then $\alpha$ and $\beta$ are incomparable. Let $\hat{P_1}=\{\alpha\,|\,\cl{\alpha}\in P_1\}$
and $\hat{P_2}=\{\alpha\,|\,\cl{\alpha}\in P_2\}$. However, $\hat{P_1}\cup\hat{P_2}\not=(\cl{m},\cl{m\oplus m})$
because it does not include the mesh patterns $\alpha$ with ${\cl{\alpha}=\cl{m}\oplus \cl{m}}$.

There are exactly two occurrences of $m$ in $m\oplus m$. These are $\eta_1$ the first~$|m|$ letters
and $\eta_2$ the last $|m|$ letters. Note that each shaded box of~${m\oplus m}$ is used by at least one
of $\eta_1$ and $\eta_2$, so if we deshade a box the resulting pattern $x$ contains at most one occurrence of $m$,
either the first or last $|m|$ letters. Let $Q_1$ and $Q_2$ be sets of patterns with underlying permutation
$\cl{m}\oplus \cl{m}$ where the first and last $|m|$ letters are the only occurrence of $m$, respectively. So any
element $Q_1$ cannot contain an element in $P_2\cup Q_2$ and similarly any element of $Q_2$ cannot
contain an element of ${P_1\cup Q_1}$. Therefore, $P_1\cup Q_1$ and $P_2\cup Q_2$ are disconnected
nontrivial components of $[m,m\oplus m]$.
\end{proof}
\end{lem}
\begin{cor}
If $m$ is skew-indecomposable, $(|m|,0),(0,|m|) \not\in\sh{m}$ and $\dim m>1$, then
$[m,m\ominus m]$ is strongly disconnected.
\end{cor}

Using Lemma 4.2 in \cite{McSt13} it is shown that
almost all intervals of the classical permutation poset are not shellable. The
proof of this follows from the Marcus-Tardos theorem. We have seen this result
does not apply in the mesh pattern case, so we cannot prove a similar result
using this technique.  A similar problem was studied for boxed mesh patterns in
permutations in~\cite{AKV13}, which is equivalent to boxed mesh patterns in
fully shaded mesh patterns. So we present the following open question:

\begin{que}
What proportion of intervals of $\mathcal{M}$ are shellable?
\end{que}

The M\"obius function in the permutation poset can be computed more easily by decomposing the
permutations into smaller parts using the direct sum, or skew-sum, see \cite{BJJS11,McSt13}. Which
leads to the following question:
\begin{que}
Can a formula for the M\"obius function of $\mathcal{M}$ be obtained by decomposing mesh patterns
using direct sums and skew sums?
\end{que}

%\bibliographystyle{alpha}
%\bibliography{bibfile}

\end{document}